\documentclass[titlepage, 11pt]{article}
\usepackage[margin=2.5cm]{geometry}
\usepackage{csquotes}
\usepackage{amsfonts}
\usepackage{amsmath}
\usepackage{amssymb}
\usepackage{amsthm}
\usepackage{authblk}
\usepackage{complexity}

\usepackage[english]{babel}
\usepackage{enumerate}
\usepackage{footnote}
\usepackage{xcolor}
\definecolor{linkB}{RGB}{4, 106, 143}
\usepackage[colorlinks=true, linkcolor={linkB}, urlcolor={linkB}, citecolor={linkB}]{hyperref}
\usepackage{latexsym, graphicx}
\usepackage{lmodern}
\usepackage{caption}
\usepackage{subcaption}
\DeclareCaptionType{algorithm}
\newcommand{\sset}[1]{\left\{#1\right\}}

\usepackage{marvosym}

\usepackage{amsmath, amssymb, amsthm}
\usepackage{thmtools}
\usepackage{thm-restate}

\declaretheorem[name=Lemma, numberwithin = section]{lemma}
\declaretheorem[name=Theorem,sibling = lemma]{theorem}

\usepackage[noabbrev,capitalise,nameinlink]{cleveref}
\addto\extrasenglish{}
\addto\extrasenglish{}
\crefname{theorem}{Theorem}{Theorems}
\crefname{proposition}{Proposition}{Propositions}
\crefname{lemma}{Lemma}{Lemmas}
\crefname{claim}{Claim}{Claims}
\crefname{subclaim}{Sub-Claim}{Sub-Claims}
\crefname{observation}{Observation}{Observations}
\crefname{remark}{Remark}{Remarks}
\crefname{corollary}{Corollary}{Corollaries}
\crefname{definition}{Definition}{Definitions}
\crefname{conjecture}{Conjecture}{Conjectures}
\crefname{question}{Question}{Questions}
\crefformat{equation}{(#2#1#3)}
\Crefformat{equation}{Equation #2(#1)#3}

\usepackage[backend=biber,style=numeric, 
maxbibnames=20, 
url=false, 
isbn=false, 
doi=false]{biblatex}
\DeclareLabelalphaTemplate{
  \labelelement{
    \field[final]{shorthand}
    \field{label}
    \field[strwidth=3,strside=left,ifnames=1]{labelname}
    \field[strwidth=1,strside=left]{labelname}
  }
  \labelelement{
    \field{year}    
  }
}

\newbibmacro{string+doiurlisbn}[1]{%
  \iffieldundef{doi}{%
    \iffieldundef{url}{%
      \iffieldundef{isbn}{%
        \iffieldundef{issn}{%
          #1%
        }{%
          \href{http://books.google.com/books?vid=ISSN\thefield{issn}}{#1}%
        }%
      }{%
        \href{http://books.google.com/books?vid=ISBN\thefield{isbn}}{#1}%
      }%
    }{%
      \href{\thefield{url}}{#1}%
    }%
  }{%
    \href{http://dx.doi.org/\thefield{doi}}{#1}%
  }%
}

\DeclareFieldFormat{title}{%
  \usebibmacro{string+doiurlisbn}{\mkbibemph{#1}}}
\DeclareFieldFormat
  [article,inbook,incollection,inproceedings,patent,thesis,unpublished]
  {title}{\usebibmacro{string+doiurlisbn}{\mkbibquote{#1\isdot}}}
\DeclareFieldFormat
  [suppbook,suppcollection,suppperiodical]
  {title}{\usebibmacro{string+doiurlisbn}{#1}}

\title{The sandwich problem for odd-hole-free and even-hole-free graphs}
\author[$^{\dagger}$]{Kathie Cameron\thanks{We acknowledge the support of the Natural Sciences and Engineering Research Council of Canada (NSERC), [funding reference number RGPIN-2016-06517]. Cette recherche a \'et\'e financ\'ee par le Conseil de recherches en sciences naturelles et en g\'enie du Canada (CRSNG), [num\'ero de r\'ef\'erence RGPIN-2016-06517]}}
\author[$^{\ddagger}$]{Aristotelis Chaniotis}
\author[$^{\mathsection}$]{Celina M. H. de Figueiredo\thanks{Supported by Conselho
Nacional de Desenvolvimento Cient\'ifico e Tecnol\'ogico of Brazil (CNPq) grant 305356/2021-6.}}
\author[$^{\ddagger}$]{Sophie Spirkl\thanks{We acknowledge the support of the Natural Sciences and Engineering Research Council of Canada (NSERC), [funding reference number RGPIN-2020-03912]. Cette recherche a \'et\'e financ\'ee par le Conseil de recherches en sciences naturelles et en g\'enie du Canada (CRSNG), [num\'ero de r\'ef\'erence RGPIN-2020-03912]. This project was funded in part by the Government of Ontario. This research was conducted while Spirkl was an Alfred P. Sloan Fellow.}}

\affil[$^{\dagger}$]{Department of Mathematics, Wilfrid Laurier University, \protect\\ Waterloo, Ontario, N2L 3C5, Canada \protect \vspace{0.25cm}}
\affil[$^{\ddagger}$]{Department of Combinatorics and Optimization, University of Waterloo, \protect\\ Waterloo, Ontario, N2L 3G1, Canada \protect \vspace{0.25cm}}
\affil[$^{\mathsection}$]{COPPE, Universidade Federal do Rio de Janeiro, \protect\\ Rio de Janeiro, 21941-918, Brazil}

\date{\today}

\begin{document}

\maketitle

\fontsize{12}{16}\selectfont

\begin{abstract}
For a property $\mathcal{P}$ of graphs, the $\mathcal{P}$-\textsc{Sandwich-Problem}, introduced by Golumbic and Shamir (1993), is the following: Given a pair of graphs $(G_1, G_2)$ on the same vertex set $V$, does there exist a graph $G$ such that $V(G)=V$, $E(G_{1})\subseteq E(G) \subseteq E(G_{2})$, and $G$ satisfies $\mathcal{P}$? A {\em hole} in a graph is an induced subgraph which is a cycle of length at least four. An odd (respectively even) hole is a hole of odd (respectively even) length. Given a class of graphs $\mathcal{C}$ and a graph $G$ we say that $G$ is {\em $\mathcal{C}$-free} if it contains no induced subgraph isomorphic to a member of $\mathcal{C}$. In this paper we prove that if $\mathcal{P}$ is the property of being odd-hole-free or the property of being even-hole-free, then the $\mathcal{P}$-\textsc{Sandwich-Problem} is $\NP$-hard.
\end{abstract}

\section{Introduction} \label{sec:int}
All graphs in this paper are finite and simple. Let $G$ be a graph. $G^c$ denotes the \emph{complement} of $G$, obtained from $G$ by replacing each edge with a non-edge and vice versa. 

Let $G_1 = (V_1, E_1)$ and $G_2 = (V_2, E_2)$. Then $G_2$ is a \emph{supergraph} of $G_1$ if $V_1 = V_2 = V$ and $E_1 \subseteq E_2$. A pair $(G_1, G_2)$ of graphs so that $G_2$ is a supergraph of $G_1$ is called a \emph{sandwich instance}. The edges in $E_1$ are called \emph{forced}, while the edges in $E_2 \setminus E_1$ are \emph{optional}. 
We let $E_3$ be the set of all edges in the complete graph with vertex set $V$ which are not in $E_2$, and call them \emph{forbidden} edges.
A graph $G$ is called a \emph{sandwich graph} for the sandwich instance $(G_1, G_2)$ if $G_2$ is a supergraph of $G$ and $G$ is a supergraph of $G_1$. 

For a property $\mathcal{P}$ of graphs, the $\mathcal{P}$-\textsc{Recognition-Problem} is the problem of deciding whether a given graph $G$ satisfies $\mathcal{P}$. The $\mathcal{P}$-\textsc{Sandwich-Problem} is the following: For a given sandwich instance $(G_1, G_2)$, does there exist a sandwich graph $G$ for $(G_1, G_2)$ so that $G$ satisfies $\mathcal{P}$?  This generalization of the recognition problem was introduced by Golumbic and Shamir \cite{golumbic}. The sandwich problem becomes the recognition problem when $G_1 = G_2$, and thus, if the $\mathcal{P}$-\textsc{Recognition-Problem} is $\NP$-hard, so is the $\mathcal{P}$-\textsc{Sandwich-Problem}. 

Let $\mathcal{P}$ be a graph property. We define the complementary property $\mathcal{P}^c$ by saying that $G$ satisfies $\mathcal{P}^c$ if and only if $G^c$ satisfies $\mathcal{P}$. 
Note that the $\mathcal{P}$-\textsc{Sandwich-Problem} is invariant under taking complements 
in the following sense. 

\begin{lemma}
\label{lem:complements}
The $\mathcal{P}^{c}$-\textsc{Sandwich-Problem} is $\NP$-hard if and only if the 
$\mathcal{P}$-\textsc{Sandwich-Problem} is. 
\end{lemma}

\begin{proof}[Proof of \autoref{lem:complements}]
An instance $(G_1, G_2)$ is a \textsc{Yes} instance for $\mathcal{P}^{c}$-\textsc{Sandwich-Problem} if and only if $(G_2^c, G_1^c)$ is a \textsc{Yes} instance for the $\mathcal{P}$-\textsc{Sandwich-Problem}.  
\end{proof}

We let $P_k$ denote an induced path on $k$ vertices, and $C_k$ denote an induced cycle on $k$ vertices. 
An induced cycle $C_k$ with $k \geq 4$ vertices is called a \emph{hole}; it is called an \emph{odd hole} if $k$ is odd, and an \emph{even hole} if $k$ is even. An \emph{antihole} is the complement of a hole. It is an \emph{odd antihole} if its complement is an odd hole, and an \emph{even antihole} otherwise. Let $\mathcal{C}$ be a set of graphs. We say that $G$ is \emph{$\mathcal{C}$-free} if no induced subgraph of $G$ is isomorphic to a graph in $\mathcal{C}$. 

A graph is \emph{Berge} if it contains no odd hole and no odd antihole as an induced subgraph. A graph $G$ is \emph{perfect} if for each induced subgraph $H$ of $G$, the clique number of $H$ equals the chromatic number of $H$. The Strong Perfect Graph Theorem \cite{chudnovsky4}, first conjectured in \cite{berge}, states that a graph is perfect if and only if it is Berge. 
It is known that Berge graphs can be recognized in polynomial time \cite{chudnovsky3}.
Research has focused on the $\mathcal{P}$-\textsc{Sandwich-Problem} for subclasses $\mathcal{P}$ of perfect graphs, and for decompositions related to perfect graphs. 
The complexity of the \textsc{Perfect-Graph}-\textsc{Sandwich-Problem} remains one of the most prominent open questions in this area~\cite{golumbic2}.

This paper is organized as follows. In \autoref{sec:odd}, we show that the hardness reduction that proves that the $C_5$-\textsc{Free}-\textsc{Sandwich-Problem} is $\NP$-hard~\cite{dantas} actually proves that the 
\textsc{Odd-Hole-Free}-\textsc{Sandwich-Problem} is $\NP$-hard. In \autoref{sec:even}, we modify the hardness reduction that proves that the \textsc{Chordal-Sandwich-Problem} is $\NP$-hard~\cite{bodlaender} to prove that 
the \textsc{Even-Hole-Free}-\textsc{Sandwich-Problem} is $\NP$-hard. 
In \autoref{sec:conc}, we give our concluding remarks. 

\section{The Odd-Hole-Free-Sandwich-Problem} 
\label{sec:odd}
We show that the hardness reduction given in~\cite{dantas} that proves that the $C_5$-\textsc{free}-\textsc{Sandwich-Problem} is $\NP$-hard actually 
proves that the \textsc{Odd-Hole-Free}-\textsc{Sandwich-Problem} is $\NP$-hard.
For the convenience of the reader, in the present section, we use the same notation of~\cite{dantas}. 

\begin{figure}[ht]
    \centering
    \includegraphics[scale=1.5, page=4]{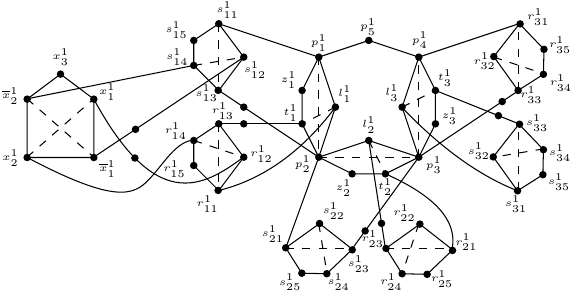}
    \caption{This figure illustrates the construction that we discuss in the proof of \autoref{thm:odd}. In the case which is illustrated the clause contains the literal $\overline{x}_{1}$. Edges represented with dashed line segments are optional edges, that is,
    edges of $E_{2} \setminus E_1$;
    the solid edges are forced, that is, edges of $E_{1}$. The omitted edges are forbidden, that is, edges of $E_{3}$.}
    \label{fig:c5-free}
\end{figure}

\begin{theorem}
\label{thm:odd} 
The \textsc{Odd-Hole-Free}-\textsc{Sandwich-Problem} is $\NP$-complete.
\end{theorem}

\begin{proof}[Proof of \autoref{thm:odd}]
First we briefly describe a construction given in~\cite[see Theorem 2 and Figure 5]{dantas} that establishes that the $C_5$-\textsc{free}-\textsc{Sandwich-Problem} is $\NP$-hard.
Let $(U,C)$ be a \textsc{3-sat} instance with a set $U = \sset{x_1, \dots, x_n}$ of variables and a set $C = \sset{c_1, \dots, c_m}$ of clauses such that each clause contains exactly three variables.
For every variable $x_i$ of $(U, C)$, we will define a set $X_i$, which consists of a five-cycle of forced edges of $E_1$, along with two optional edges of $E_2$, which are chords of this five-cycle and form a matching. At least one of these optional edges is present in every $C_5$-free sandwich graph for our instance, and which of the chords is present will correspond to whether $x_i$ is true or false. For every clause $c_j$ of $(U, C)$, and every literal $\ell_q^j \in \sset{x_i, \overline{x_i}}$ in $c_j$, we add two gadgets $\sset{r^j_{q1}, \dots, r^j_{q5}}$ and $\sset{s^j_{q1}, \dots, s^j_{q5}}$, 
each a five-cycle of forced edges along with two optional edges, which are designed to provide a copy of $x_i$ if $x_i$ is true and if $x_i$ is false, respectively\footnote{We remark that the auxiliary vertices with no label are in a two-edge path needed to connect the gadgets.}.
For $c_j$, we add a five-cycle with vertex set 
$\sset{p^j_1, \dots, p^j_5}$, which has one optional edge of $E_2$ for each literal, and two forced edges of $E_1$, we also add the vertices
$\cup_{q\in \{1,2,3\}}\{l_{q}^{j},t_{q}^{j},z_{q}^{j}\}$, and finally for each $q\in\{1,2,3\}$, we add forced edges so that $p_{q}^{j}, z_{q}^{j}, t_{q}^{j}, p_{q+1}^{j}, l_{q}^{j}$ is a five-cycle of forced edges, and we add the optional edge $t_{q}^{j}l_{q}^{j}$ in $E_{2}$. It follows that if there is a set of optional edges that we can add to create a $C_5$-free sandwich graph $G$, then there is a truth assignment in which for every clause $c_j$, not all three optional edges among  $\sset{p^j_1, \dots, p^j_5}$ are present in $G$, and hence $c_j$ contains a true literal. 

It is clear from the construction that for every \textsc{3-SAT} instance in which each clause contains exactly three variables, the graph $G_1$ of forced edges is triangle-free. The graph of optional edges is a forest, and each connected component consists either of an edge, a single vertex, or a three-edge path. So every triangle in $G_2$ contains exactly one optional edge and two forced edges.

\begin{figure}[ht]
    \centering
    \includegraphics[page=5, scale=1]{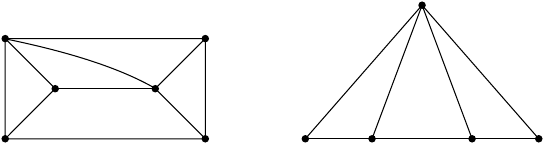}
    \caption{On the left-hand side the graph $P_{6}^{c}$ and on the right-hand side the graph $\{P_{4}+P_{1}\}^{c}$.}
    \label{fig:p6.compl}
\end{figure}

Every triangle in $G_2$ shares exactly one edge with another triangle in $G_2$, and this edge is a forced edge shared by exactly two triangles, and hence the graph $G_{2}$ does not contain as a (non-induced) subgraph the graph $\{P_{4}+P_{1}\}^{c}$, because the latter contains a triangle sharing two of its edges with other triangles (see \autoref{fig:p6.compl}).
Since $P^{c}_{6}$ contains $\{P_{4}+P_{1}\}^{c}$ as an induced subgraph, by the above observation, it follows that every sandwich graph $G$ for the constructed instance $(G_1, G_2)$ is $P^c_6$-free.
Hence, every sandwich graph $G$ for the constructed instance $(G_1, G_2)$  is $C^c_k$-free, for $k \geq 7$.
Therefore, the \textsc{Odd-Antihole-Free}-\textsc{Sandwich-Problem} is $\NP$-hard.
By \autoref{lem:complements}, we obtain that the \textsc{Odd-Hole-Free}-\textsc{Sandwich-Problem} 
is $\NP$-hard.
\end{proof}

\newcommand{\bsl}{\backslash}
\newcommand{\gp}{G^\prime}
\newcommand{\ep}{E^\prime}
\newcommand{\xbar}{{\overline{X}}}
\newcommand{\ybar}{{\overline{Y}}}
\newcommand{\zbar}{{\overline{Z}}}
\newcommand{\abar}{{\overline{a}}}
\newcommand{\Abar}{{\overline{A}}}
\newcommand{\Bbar}{{\overline{B}}}
\newcommand{\Cbar}{{\overline{C}}}
\newcommand{\ione}{${\cal I}_1$\ }
\newcommand{\itwo}{${\cal I}_2$\ }
\newcommand{\al}{($\alpha$)}
\newcommand{\be}{($\beta$)}
\newcommand{\ga}{($\gamma$)}
\newcommand{\de}{($\delta$)}
\newcommand{\uset}{{\cal U}}
\newcommand{\aset}{{\cal A}}
\newcommand{\bset}{{\cal B}}

\section{The Even-Hole-Free-Sandwich-Problem} 
\label{sec:even}

For the convenience of the reader, in the present section, we use the same notation of~\cite{bodlaender}.
We shall modify a construction given in~\cite{bodlaender} that establishes that the \textsc{Chordal-Sandwich -Problem} is $\NP$-hard in order to prove the following:

\begin{theorem}
\label{thm:even} 
The \textsc{Even-Hole-Free-Sandwich-Problem} is $\NP$-hard.
\end{theorem}

\begin{proof}[Proof of \autoref{thm:even}]
First we briefly describe a construction given in~\cite{bodlaender} that establishes that the chordal graph sandwich problem is $\NP$-hard.
Let $(U,C)$ be a \textsc{3-SAT} instance with a set $U$ of variables and a set $C$ of clauses such that each clause contains exactly three variables. 
We assume that no clause contains both a variable and its complement. 
We have a vertex $H$ called the head and a vertex $F$  called the foot.
For each variable $X$, we have two vertices $S_X$ and $S_{\overline{X}}$, called {\em shoulders}.
For each variable $X$ and each clause $i$ containing either $X$ or ${\overline{X}}$, we have two vertices $K_X^i$ and $K_{\overline{X}}^i$, called {\em knees}.
These are the vertices in the vertex set $V$ of the constructed instance.

\begin{figure}[h]
    \centering
    \includegraphics[page=1, scale=1.25]{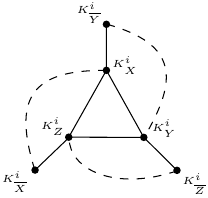}
    \caption{A gadget which corresponds to the clause $(X+Y+Z)$ in the construction 
    from~\cite{bodlaender}. Solid edges are forced, dashed edges are forbidden, and omitted edges are optional.}
    \label{fig:even_holes.clause}
\end{figure}

The forbidden edges in $E_3$ are $HF$, $S_XS_{\overline{X}}$, $K_X^iK_{\overline{X}}^i$, for $X\in U$ and clauses $i$ containing $X$ or $\overline{X}$, so $E_3$ is a matching.
For each variable $X$ and each clause $i$ containing either $X$ or ${\overline{X}}$, 
we will have a six-cycle $H, S_X, K_{\overline{X}}^i, F, K_X^i, S_{\overline{X}}$
of forced edges in $E_1$. For each clause $i = (X + Y + Z)$, we will have a forced 3-sun with triangle $\sset{K_X^i, K_Y^i, K_Z^i}$, and forced pendant edges $K_{\overline{Y}}^iK_X^i, K_{\overline{Z}}^iK_Y^i, K_{\overline{X}}^iK_Z^i$ (see \autoref{fig:even_holes.clause}). The vertices $K_X^i, K_Y^i, K_Z^i$ are called {\em active knees}, and the vertices $K_{\overline{X}}, K_{\overline{Y}}^i, K_{\overline{Z}}^i$ are called {\em inactive knees}. All remaining edges  in the complete graph with vertex set $V$ which are not in $E_1 \cup E_3$ are optional edges. In what follows we refer to the graphs $(V,E_{1})$ and $(V,E_{2})$ as $G_{1}$ and $G_{2}$ respectively.

As shown in \cite{bodlaender}, there is a satisfying truth assignment for $(U,C)$ if and only if there is a sandwich graph $G$ for the sandwich instance $(G_{1}, G_{2})$, such that $G$ is chordal. As noted in \cite{bodlaender}, there are exactly two ways to add optional edges to the graph $G_{1}$ in order to triangulate each forced six-cycle corresponding to a variable: the {\em positive} orientation adds $HK_X^i, K_X^iS_X, S_XF$ or the {\em negative} orientation adds $HK_{\overline{X}}^i, K_{\overline{X}}^iS_{\overline{X}}, S_{\overline{X}}F$. Further, all six-cycles corresponding to variable $X$ must be oriented in the same way, since otherwise the set $\{H,S_{X}, S_{\overline{X}}, F\}$ induces a four-cycle. Therefore, given a chordal completion $G$ we may obtain a truth assignment for $(U,C)$, and a truth assignment for $(U,C)$ {\em indicates a set of optional edges} that can be added to the sandwich instance in order to triangulate each forced six-cycle corresponding to a variable.

Suppose that the sandwich instance described above has a chordal completion $G$. Then the truth assignment for $(U,C)$ which we obtain by $G$, as described above, is satisfying ~\cite[see Lemma 2]{bodlaender}.

Conversely, suppose that we have a satisfying truth assignment for $(U,C)$. Then we add to the corresponding sandwich instance the set of optional edges which is indicated by this truth assignment in order to triangulate each forced six-cycle. To finish the chordal completion ~\cite[see Lemma 2]{bodlaender}, the following sets of optional edges are added: 
\begin{itemize}
    \item true shoulders that is, shoulders $S_{X}$ where $X$ is true and shoulders $S_{\overline{X}}$ where $\overline{X}$ is false, form a clique;
    \item true knees that is, knees $K_{X}^{i}$ where $X$ is true and knees $K^i_{\overline{X}}$ where $X$ is false, form a clique; and
    \item each true shoulder is adjacent to every knee.
\end{itemize}

Next we show how to modify the construction given in~\cite{bodlaender} which we described above in order to establish that the \textsc{Even-Hole-Free-Sandwich-Problem} is $\NP$-hard.
We add two auxiliary vertices $W_1$ and $W_2$ to the vertex set $V$, obtaining $V'= V \cup \{W_1, W_2\}$.
We add three forced edges to the forced edge set $E_1$, such that we have a four-vertex path induced by $\sset{H, W_1, W_2, F}$,
obtaining $E'_1 = E_1 \cup \{HW_1, W_1W_2, W_2F\}$. All other edges incident to vertices $W_{1}, W_{2}$ are forbidden edges in $E'_3$. In what follows  we refer to the graphs $(V',E'_{1})$ and $(V',E'_{2})$ as $G_{1}'$ and $G_{2}'$, respectively.

\begin{figure}
    \centering
    \includegraphics[page=3, scale=1.5]{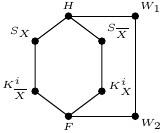}
    \caption{Adding a four-vertex path in the construction from ~\cite{bodlaender}. Both the vertices $W_{1}$ and $W_{2}$ have degree two in $G_{2}'$.}
    \label{fig:even_holes.adding.P_4}
\end{figure}

We show that there is a satisfying truth assignment for $(U,C)$ if and only if there is a sandwich graph $G$ for $(G_{1}',G_{2}')$, such that $G$ is even-hole-free. 

Suppose there is a satisfying truth assignment for $(U,C)$. Consider the sandwich graph $G$ obtained by adding the optional edges added in~\cite{bodlaender} as described above. Recall that the subgraph of $G$ induced by $V= V' \setminus \{W_1, W_2\}$ is a chordal graph.
Suppose $G$ contains an even hole $B$. Then $B$ contains both auxiliary vertices $W_1$ and $W_2$, and also vertices $H$ and $F$.
If the hole $B$ contains a true shoulder $s$, then since $s$ is adjacent in $G$ to $H$ and $F$, we have that $|B| = 5$.
Otherwise, if the hole $B$ contains a true knee $k$, then since $k$ is adjacent to $H$ and $F$, again we have that $|B| = 5$. Therefore, $B$ contains a path from $H$ to $F$ using only the non-true shoulders and knees. However, the set of non-true shoulders and non-true knees is an independent set, with shoulders only adjacent to $H$, and knees only adjacent to $F$, so no such path exists.
We conclude that $G$ is even-hole-free.

Suppose there is a sandwich graph $G$ that is even-hole-free.
The six-hole induced by $\{H, W_1, W_2, F, K_X^i, S_{\overline{X}}\}$ in $G_{1}'$ implies that the sandwich graph $G$ contains at least one of the optional edges $HK_X^i$ or $FS_{\overline{X}}$. If $HK_X^i \in E(G)$, then $HK_{\overline{X}}^i \not \in E(G)$, as otherwise we have a $C_4$ induced by $\{H, K_X^i, F, K_{\overline{X}}^i\}$ in $G$. We claim that $FS_X \in E(G)$. Suppose not, then the set $\{H, W_1, W_2, F, K_{\overline{X}}^i, S_X\}$ induces a six-hole in $G$, which is a contradiction. Hence, $FS_X \in E(G)$.
We conclude that if $HK_X^i \in E(G)$, then $FS_X \in E(G)$, which in turn implies $S_XK_X^i  \in E(G)$, as otherwise we have a $C_4$ induced by $\{H, K_X^i, F, S_X\}$. We call the three added optional edges $HK_X^i, FS_X, S_XK_X^i$ the {\em positive} orientation. Otherwise $FS_{\overline{X}} \in E(G)$, and then $FS_X \not \in E(G)$, as otherwise we have a $C_4$ induced by $\{H, F, S_X, S_{\overline{X}}\}$. Now the six-hole induced by $\{H, W_1, W_2, F, K_{\overline{X}}^i, S_X\}$ implies that $HK_{\overline{X}}^i \in E(G)$.
We conclude that if $FS_{\overline{X}} \in E(G)$, then $HK_{\overline{X}}^i \in E(G)$, which in turn implies $S_{\overline{X}}K_{\overline{X}}^i \in E(G)$, as otherwise we have a $C_4$ induced by $\{H, F, S_{\overline{X}}, K_{\overline{X}}^i\}$. We call the three added optional edges $FS_{\overline{X}}, HK_{\overline{X}}^i, S_{\overline{X}}K_{\overline{X}}^i$ the {\em negative} orientation.
All six-cycles corresponding to variable $X$ must be oriented in the same way, which defines the truth assignment of $X$ to be true or false.

\begin{figure}[ht]
    \centering
    \includegraphics[page=2, scale=1.25]{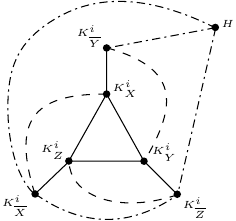}
    \caption{If $i = (X + Y + Z)$ is a clause with three negative literals and $G$ the corresponding sandwich graph, then the four-cycle $\{H, K_{\overline{Z}}^i,F,K_{\overline{X}}^i\}$ which is formed by the optional edges which 
   are assumed to be present in $G$
implies that the optional edge $K_{\overline{Z}}^iK_{\overline{X}}^i$ 
    must be present in $G$ as well. Solid edges are forced, dashed edges are forbidden, and dashed-dotted edges are the optional edges 
    known so far to be present in $G$.}
    
     \label{fig:even_holes_finding.a.C4}
\end{figure}

To prove that we have defined a satisfying truth assignment, assume that there is a clause $i = (X + Y + Z)$ with three negative literals.
The four-cycle $\{H, K_{\overline{Z}}^i,F,K_{\overline{X}}^i\}$ which is formed by the optional edges which 
are assumed to be present in $G$
implies that the optional edge $K_{\overline{Z}}^iK_{\overline{X}}^i$ 
must be present in $G$ as well 
(see \autoref{fig:even_holes_finding.a.C4}). By symmetry, $K_{\overline{X}}^i K_{\overline{Y}}^i \in E(G)$.
The four-cycle induced by $\{K_{\overline{Z}}^i, K_Y^i, K_Z^i, K_{\overline{X}}^i\}$ implies the optional edge $K_Y^i K_{\overline{X}}^i \in E(G)$.
We conclude that $G$ contains a four-hole induced by $\{K_X^i, K_Y^i, K_{\overline{X}}^i, K_{\overline{Y}}^i\}$, which is a contradiction.
\end{proof}

\section{Concluding remarks}  \label{sec:conc}
Almost monotone properties were introduced in~\cite{chudnovsky}, where the sandwich problem was proved to be in $\P$ in case the recognition problem was known to be in $\P$. The properties of containing an odd hole and of containing an even hole were proved to be almost monotone.
At that time, even-hole-free graphs and Berge graphs were known to be recognized in polynomial time~\cite{conforti, chudnovsky3}, but the  polynomial-time recognition of odd-hole-free graphs was established later~\cite{chudnovsky2}. 

In the \textsc{Not}-$\mathcal{C}$-\textsc{Free}-\textsc{Sandwich-Problem}, we are asking if there exists a sandwich graph in which there exists an induced subgraph isomorphic to a graph in $\mathcal{C}$, whereas in the $\mathcal{C}$-\textsc{Free}-\textsc{Sandwich-Problem}, we are testing if there exists a sandwich graph $G$ such that for every induced subgraph $H$ of $G$, $H$ is not in $\mathcal{C}$. The latter problem has an additional alternation, which is an indication that the \textsc{Not}-$\mathcal{C}$-\textsc{Free}-\textsc{Sandwich-Problem} might always be easier than the $\mathcal{C}$-\textsc{Free}-\textsc{Sandwich-Problem}. Clearly, if the $\mathcal{C}$-\textsc{Free}-\textsc{Recognition-Problem} is $\NP$-hard (e.\ g.\ if $\mathcal{C}$ is the set of prisms), then the \textsc{Not}-$\mathcal{C}$-\textsc{Free}-\textsc{Sandwich-Problem} is $\NP$-hard. This leads to two open questions: 
\begin{itemize}
\item Is there a class of graphs $\mathcal{C}$ such that the $\mathcal{C}$-\textsc{Free}-\textsc{Recognition-Problem} is in $\P$, but the \textsc{Not}-$\mathcal{C}$-\textsc{Free}-\textsc{Sandwich-Problem} is $\NP$-hard?
\item Is there a set $\mathcal{C}$ such that the $\mathcal{C}$-\textsc{Free}-\textsc{Sandwich-Problem} is in $\P$, but the \textsc{Not}-$\mathcal{C}$-\textsc{Free}-\textsc{Sandwich-Problem} is $\NP$-hard?
\end{itemize}
 
Truemper configurations \cite{truemper}, such as prisms, thetas, and pyramids, were considered for the \textsc{Not}-$\mathcal{C}$-\textsc{Free}-\textsc{Sandwich-Problem}. Both the \textsc{Prism}-\textsc{Free}-\textsc{Sandwich-Problem} and \text{Not}-\textsc{Prism}-\textsc{Free}-\textsc{Sandwich-Problem} are $\NP$-hard, because the \textsc{Prism}-\textsc{Free}-\textsc{Recogni-tion-Problem} is $\NP$-hard~\cite{maffray}. However, the \textsc{Theta}-\textsc{Free}-\textsc{Sandwich-Problem} is $\NP$-hard \cite{dantas}, but the \textsc{Not-Theta-Free-Sandwich-Problem} is in $\P$~\cite{chudnovsky}. The \textsc{Not-Pyramid-Free-Sandwich-Problem} is in $\P$~\cite{chudnovsky}, but the complexity of the \textsc{Pyramid-Free-Sand-wich-Problem} remains open. The complexity of the \textsc{Perfect-Graph-Sandwich}-\textsc{Prob-lem} 
remains one of the most prominent open questions in this area~\cite{golumbic2}.

\printbibliography

\end{document}